\documentclass[11pt]{amsart}
\usepackage{amsmath}
\usepackage{amsthm}
\usepackage{amsfonts}
\usepackage{amssymb}
\usepackage[all]{xy}
\usepackage{alltt}
\usepackage{enumitem}
\usepackage{stmaryrd}
\usepackage{comment}
\usepackage{xspace}
\usepackage{comment}
\usepackage[
 colorlinks=true,urlcolor=blue,linkcolor=blue,citecolor=blue
]{hyperref}

\theoremstyle{plain}
\newtheorem{prop}{Proposition}[section]

\newtheorem{thm}[prop]{Theorem}
\theoremstyle{definition}

\newtheorem{cor}[prop]{Corollary}

\newcommand{\Z}{\ensuremath{\mathbf{Z}}}
\newcommand{\F}{\ensuremath{\mathbf{F}}}

\newcommand{\Q}{\ensuremath{\mathbf{Q}}}

\newcommand{\C}{\ensuremath{\mathbf{C}}}
\newcommand{\GL}{\ensuremath{\mathbf{GL}}}
\newcommand{\SL}{\ensuremath{\mathbf{SL}}}

\newcommand{\PSL}{\ensuremath{\mathbf{PSL}}}
\newcommand{\AGL}{\ensuremath{\mathbf{AGL}}}
\newcommand{\Sym}{\ensuremath{\mathbf{S}}}
\newcommand{\A}{\ensuremath{\mathbf{A}}}
\newcommand{\Hol}{\ensuremath{\mathbf{Hol}}}
\newcommand{\Aut}{\ensuremath{\mathbf{Aut}}}
\newcommand{\UC}{\ensuremath{\mathbf{UC}}}

\newcommand{\n}{\noindent}

\newcommand{\dih}{\ensuremath{\mathbf{D}}}
\newcommand{\cyc}{\ensuremath{\mathbf{C}}}
\newcommand{\upt}{\ensuremath{\mathbf{U}}}
\renewcommand{\F}{\ensuremath{\mathbf{F}}}
\renewcommand{\Z}{\ensuremath{\mathbf{Z}}}
\renewcommand{\C}{\ensuremath{\mathbf{C}}}

\renewcommand{\A}{\ensuremath{\mathbf{A}}}
\newcommand{\D}{\ensuremath{\mathbf{D}}}

\newcommand{\mm}[4]{\ensuremath{\begin{bmatrix} #1 & #2 \\ #3 & #4 \end{bmatrix}}}
\usepackage[margin=1.25in]{geometry}

\usepackage[
        colorlinks=true,urlcolor=blue,linkcolor=blue,citecolor=blue
]{hyperref}

\begin{document}  
 
\title{Fuchs' problem for linear groups}
\date{\today}
 
\author{Keir Lockridge} 
\address {Department of Mathematics \\
Gettysburg College \\
Gettysburg, PA 17325, USA}
\email{klockrid@gettysburg.edu}

\author{Jacinda Terkel} 
\address {Department of Mathematics \\
Gettysburg College \\
Gettysburg, PA 17325, USA}
\email{terkja01@gettysburg.edu}

\thanks{The second author was supported by the Julia Gatter Hirsch Endowed Fund for Faculty/Student Research in the Mathematics Department at Gettysburg College.}

 
\begin{abstract}
Which groups can occur as the group of units in a ring? Such groups are called {\em realizable}. Though the realizable members of several classes of groups have been determined (e.g., cyclic, odd order, alternating, symmetric, finite simple, indecomposable abelian, and dihedral), the question remains open. The general linear groups are realizable by definition: they are the units in the corresponding matrix rings. In this paper, we study the realizability of two closely related linear groups, the special linear groups and the affine general linear groups. We determine which special linear groups of degree 2 over a finite field are realizable by a finite ring, and we determine which affine general linear groups of degree 1 over a cyclic group are realizable by a finite ring. We also give partial results for certain linear groups of other degrees and for rings of characteristic zero.
\end{abstract}
 
\maketitle
\thispagestyle{empty}



\section{Introduction}

This paper is part of an ongoing investigation into a question posed (originally for abelian groups) by L\'{a}szl\'{o} Fuchs in \cite[Problem 72]{F60}: {\em which groups can occur as the group of units in a unital ring?} Such groups are called realizable. More precisely, a group $G$ is {\bf realizable} if there is a ring $R$ with identity whose unit group $R^\times$ is isomorphic to $G$. In this situation, we say that $R$ {\bf realizes} $G$. Though there is no full classification of the realizable groups, there are many partial results. For example, the realizable groups have been classified in the following families: cyclic groups (\cite{Gil63, PS70}), groups of odd order (\cite{Ditor71}), alternating and symmetric groups (\cite{DO14b}), finite simple groups (\cite{DO14a}), indecomposable abelian groups (\cite{CL15}), and dihedral groups (\cite{CL17d}). Partial results on realizable 2-groups appear in \cite{CL17} and \cite{SW20}.

The general linear groups $\GL_n(R)$ (for $R$ a commutative ring with identity) are realizable since $\GL_n(R)$ is the group of units in the ring of $n \times n$ matrices with entries in $R$. In this paper, we initiate an analysis of two closely related linear groups: the special linear groups and the affine general linear groups. The special linear group is an important normal subgroup of $\GL_n(R)$ defined by \[\SL_n(R) = \{ M \in \GL_n(R) \, | \, \det(M) = 1\}.\] Since the determinant map has a right inverse, we have \[\GL_n(R) = \SL_n(R) \rtimes R^\times\] In \S \ref{sec:SL}, we will prove:

\begin{thm} Let $F$ be a finite field. 
\begin{enumerate}
    \item If $|F|$ is even, then $\SL_2(F)$ is realizable if and only if $F = \F_2$.
    \item If $|F|$ is odd and congruent to $\pm 1$ modulo 8, then $\SL_2(F)$ is not realizable.
    \item If $|F|$ is odd and congruent to $\pm 3$ modulo 8, then $\SL_2(F)$ is not the group of units in a finite ring.
    
\end{enumerate} \label{sl2main}
\end{thm}

\noindent Note that the last two items imply that $\SL_2(F)$ is not the group of units in a finite ring when $|F|$ is odd. The remaining open question here is whether $\SL_2(F)$ is realizable in characteristic zero when $|F| \equiv \pm 3\, (8).$ To give one positive example, consider the ring of real quaternions \[ \{ a + b \,\mathbf{i} + c\,\mathbf{j} + d\,\mathbf{k} \, | \, a, b, c, d \in \mathbf{R}\}\] ($\mathbf{R}$ denotes the real numbers). The ring of Hurwitz integers is the set of quaternions where each coordinate is either an integer or half an odd integer. It is well known that the group of units in the Hurwitz integers is $\SL_2(\F_3)$. We do not have an answer for any other values of $|F|$.

Rather than restricting the general linear group to the special linear group, one might instead extend it to the affine general linear group. Viewing the elements of $\GL_n(R)$ as invertible $R$-linear endomorphisms of $R^n$, the affine general linear group is defined by \[ \AGL_n(R) = \{x \mapsto ax + b \, | \, a \in \GL_n(R), b \in R^n\} = R^n \rtimes \GL_n(R).\] When the underlying additive abelian group $(R, +)$ is cyclic, this group happens to coincide with the holomorph of $(R, +)$. The {\bf holomorph} of a group $G$ is \[ \Hol(G) = G \rtimes \Aut(G).\] We have $\AGL_n(R) \leq \Hol(R^n),$ with equality if $(R, +)$ is cyclic. (Equality rarely occurs; for example, \[\AGL_1(\F_4) = \A_4 \lneq \Sym_4 = \Hol(\F_4) = \Hol(\Z_2 \times \Z_2),\] where $\F_4$ is the field of order 4, $\A_4$ is the alternating group of degree 4, and $\Sym_4$ is the symmetric group of degree 4.) 

In \S \ref{sec:AGL}, we will prove the following theorem summarizing which groups of the form 
$\AGL_1(\Z_n) = \Hol(\Z_n)$ are the group of units in a finite ring. We use the following notation: $\Z_n$ is the additive cyclic group of order $n$; $\C_n$ is the multiplicative cyclic group of order $n$; $\Sym_n$ is the symmetric group of degree $n$; $\D_n$ is the dihedral group of order $n$; $\mathbf{M}_n(R)$ is the ring of $n \times n$ matrices with entries in $R$; and $\upt_n(R)$ is the ring of upper triangular matrices in $\mathbf{M}_n(R)$.
\begin{thm} The table below is a complete list of the groups $\Hol(\Z_n)$ which occur as the group of units in a finite ring of characteristic $c$, with an example of such a ring in each case.
\begin{center}
\begin{tabular}{| l | l | l |}
\hline

$c = 1$ & $\Hol(\Z_1)$ & $\{0\}$ \\ \hline
$c = 2$ & $\Hol(\Z_1) = \C_1$ &  $\F_2$ \\
 & $\Hol(\Z_2) = \C_2$ &   $\F_2[x]/(x^2)$ \\
 & $\Hol(\Z_3) = \Sym_3 = \D_6$ &   $\mathbf{M}_{2}(\F_2)$ \\
 & $\Hol(\Z_4) = \D_8$ &   $\upt_3(\F_2)$\\
& $\Hol(\Z_6) = \D_{12}$ &   $\F_2[\dih_6]$ \\
 & $\Hol(\Z_{12}) = \D_8 \times \D_6$ & $\upt_3(\F_2)\times \mathbf{M}_{2}(\F_2)$  \\ \hline

$c = 3$ & $\Hol(\Z_2)$ & $\F_3$ \\
& $\Hol(\Z_6)$ & $\upt_2(\F_3)$ \\ \hline

$c = 4$ & $\Hol(\Z_2)$ & $\Z_4$ \\
& $\Hol(\Z_4)$ & $\mathrm{End}_{\Z}(\cyc_4 \times \cyc_2)$ \\
& $\Hol(\Z_6)$ & $\Z_4 \times \mathbf{M}_{2}(\F_2)$ \\ 
& $\Hol(\Z_{12})$ & $\mathbf{M}_{2}(\F_2) \times \mathrm{End}_{\Z}(\cyc_4 \times \cyc_2)$ \\ \hline

$c = 6$ & $\Hol(\Z_2)$ & $\Z_6$ \\
& $\Hol(\Z_6)$ & $\F_2 \times \upt_2(\F_3)$ \\ \hline
\end{tabular}
\end{center}

\label{aglmain}
\end{thm}

\n We do not have a complete answer for rings of characteristic zero, but we are able to show that if $\Hol(\Z_n)$ is realizable in characteristic zero, then $n$ is twice an odd integer (see \S \ref{aglchar0}). Theorem \ref{aglmain} has the following interesting consequence.

\begin{thm} \label{neat}
    The following statements are equivalent for a positive integer $n$.
    \begin{enumerate}
        \item The integer $n$ is a divisor of $12.$ \label{equiv:1}
        \item The group $\Hol(\Z_n)$ is the group of units in a finite ring. \label{equiv:2}
        \item The unit group of $\Z_n[x]$ is an elementary abelian $2$-group {\em (}\cite{CM13}{\em )}. \label{equiv:3}
        \item The group $\Hol(\Z_n)$ is a direct product of dihedral groups. \label{equiv:4}
        \item The order of $\Hol(\Z_n)$ is a divisor of $48.$ \label{equiv:5}
    \end{enumerate}
\end{thm}

\n Theorem \ref{aglmain} implies that (\ref{equiv:1}) and (\ref{equiv:2}) are equivalent; (\ref{equiv:1}) and (\ref{equiv:3}) are equivalent by \cite{CM13}. We will prove that (\ref{equiv:4}) and (\ref{equiv:5}) are each equivalent to (\ref{equiv:1}) at the end of \S 3.1, after we have collected several facts about $\Hol(\Z_n)$.

\section{Special linear groups} \label{sec:SL}

First, we summarize a few well-known facts about the special linear group. For a positive prime power $q = p^k$, let $\F_q$ denote the finite field of order $q$.

\begin{prop}
Let $m$ be a positive integer and let $q = p^k$ be a positive prime power. \label{slfacts}
\begin{enumerate}
    \item The center of $\SL_m(\F_q)$ is cyclic of order $\gcd(m, q-1)$. \label{slcenter}
    \item The only proper, non-trivial, normal subgroup of $\GL_2(\F_2) = \SL_2(\F_2) = \Sym_3$ is $\C_3.$ \label{sl22normal}
    \item The proper, non-trivial, normal subgroups of  $\GL_2(\F_3)$ are $\SL_2(\F_3)$, $\Q_8$, and $\C_2$. \label{gl23normal}
    \item The proper, non-trivial, normal subgroups of $\SL_2(\F_3)$ are $\Q_8$ and $\C_2$. \label{sl23normal}
    \item For $q > 3$, the normal subgroups of $\GL_2(\F_q)$ are the subgroups of the center and the subgroups containing $\SL_2(\F_q)$. \label{gl2qnormal}
    \item For $q > 3$, the only proper, non-trivial normal subgroup of $\SL_2(\F_q)$ is its center $\C_2$. \label{sl2qnormal}
\end{enumerate}
\end{prop}

The next two propositions narrow the possible characteristics of a ring that realizes a finite special linear group.

\begin{prop}
    If $\SL_2(\F_q)$ is the group of units in a ring $R$, then the characteristic of $R$ is $2$ if $q$ is even and the characteristic of $R$ is $0, 2, 3, 4$ or $6$ if $q$ is odd. \label{sl2char}
\end{prop}
\begin{proof}
    Suppose $R$ is a ring of characteristic $c$ that realizes $G = \SL_2(\F_q)$. Note that $c \neq 1$ since $R$ cannot be the zero ring. The characteristic subring $\Z_c$ of $R$ lies in the center of $R$, and $\Z_c^\times$ lies in the center of $G$. By Proposition \ref{slfacts} (\ref{slcenter}), the center of $G$ is either trivial (when $q$ is even) or $\Z_2$ (when $q$ is odd). Since $|\Z_c^\times| = 1$ if and only if $c = 2$ and $|\Z_c^\times| = 2$ if and only if $c \in \{0,3,4,6\}$, the proof is complete.
 \end{proof}

\begin{prop}
    If $m \geq 2$ and $\gcd(q-1, m)  = 1$, then $\SL_m(\F_q)$ is realizable if and only if $q = 2$. \label{mqm1}
\end{prop}
\noindent Note that $\SL_m(\F_2)$ is only realizable in characteristic 2 by Proposition \ref{sl2char}.
\begin{proof}
    For $m \geq 2$, $\SL_m(\F_2) = \GL_m(\F_2)$ is realizable in characteristic 2 via the ring of $m \times m$ matrices with entries in $\F_2$. On the other hand, if $\gcd(q-1, m) = 1$, $\SL_m(\F_q)$ has trivial center and hence $\SL_m(\F_q) = \PSL_m(\F_q).$ With the exception of $m = 2$ and $q \in \{2,3\},$ this group is simple. We cannot have $m = 2$ and $q = 3$ since in this case $\gcd(m, q-1) \neq 1.$ Thus, if $\SL_m(\F_q)$ is realizable, then $m \geq 2$ and $q = 2$ by the classification of realizable finite simple groups in \cite{DO14a}. 
\end{proof}

Before turning our attention to finite rings, we record one more result used freely in this section.

\begin{prop}[{\cite[2.2]{CL19}}]
    If a group $G$ is realizable in characteristic $n,$ then so are all of its maximal abelian subgroups (in characteristic $n$).
\end{prop}

\subsection{Finite rings}

As an immediate consequence of Proposition \ref{mqm1}, we have the following corollary.

\begin{cor}
    The group $\SL_2(\F_{2^k})$ is realizable in characteristic $c$ if and only if $k = 1$ and $c = 2$. \label{slcor}
\end{cor}

\noindent Thus, to complete the proof of Theorem \ref{sl2main} for finite rings, we must prove that $\SL_2(F)$ is not realizable by a finite ring when $F$ has odd characteristic. Rings of characteristic zero are addressed in \S \ref{sl2char0}.

Consider $\SL_2(\F_q)$ when $q$ is an odd prime power. If such a group is realizable by a finite ring $R$, then $R$ has characteristic 2, 3, 4, or 6. For $q > 3,$ the only proper, non-trivial, normal subgroup of $\SL_2(\F_q)$ is its center $\C_2$, and since the center is not a summand, the group $\SL_2(\F_q)$ is indecomposable (it is not isomorphic to a direct product of two nontrivial groups). Further, the binary tetrahedral group $\SL_2(\F_3)$ is indecomposable. Thus, since every ring of characteristic 6 is the direct product of a ring of characteristic 2 with a ring of characteristic 3, we may assume $R$ has characteristic 2, 3, or 4.

Next, we reduce to the case $q = 3$ and then show that $\SL_2(\F_3)$ is not the group of units in a finite ring.

\begin{prop}
    Let $q$ be an odd prime power. If $\SL_2(\F_q)$ is the group of units in a finite ring, then $q = 3$ and $c = 2$ or $c = 4$.
\end{prop}
\begin{proof}
    Suppose $R$ is a finite ring with group of units $G = \SL_2(\F_q)$, where $q = p^k > 3$ is an odd prime power. By the above discussion, we may assume $R$ has characteristic $c = 2,3,$ or $4$. 

    Let $J$ denote the Jacobson radical of $R$. Since $1 + J$ is a normal subgroup of $G$, we must have $|J| \in \{1, 2, q^3 - q\}.$ On the other hand, $|J|$ must be a power of 2 or a power of 3. Since $q^3 - q = q(q^2 -1)$ is never a prime power, we must have $|J| = 1$ and $c \in \{2,3\}$ (when $R$ has characteristic 4, we have $2 \in J$), or $|J| = 2$ and $c\in\{2,4\}$.
    
    If $J$ is trivial and $c = 2$ or $c = 3$, then $R$ is isomorphic to a finite product of matrix rings of the form $\mathbf{M}_n(\F_{c^r})$ by the Artin-Wedderburn Theorem. Since $\SL_2(\F_q)$ is indecomposable, we must have \[\SL_2(\F_q) = \GL_n(\F_{c^r}).\] Equating the centers of both sides of the above equation, we obtain $\C_2 = \C_{c^r - 1}$, forcing $c^r = 3$. But now $\SL_n(\F_3)$ is a proper normal subgroup of $\SL_2(\F_q)$, forcing $\SL_n(\F_3) \leq \C_2$ and $n = 1$. This implies $\SL_2(\F_q) = \GL_1(\F_3) = \C_2$, a contradiction. 
    
    If $|J| = 2$ and $c = 2$ or $c = 4$, then $R/J$ is a ring of characteristic 2 with unit group $\PSL_2(\F_q)$. Since $q > 3$, this group is a nonabelian finite simple group, and the only realizable nonabelian finite simple groups are the groups $\PSL_m(\F_2)$ for $m \geq 3$ (\cite{DO14a}). There is an exceptional isomorphism $\PSL_2(\F_q) \cong \PSL_m(\F_2)$ only if $q = 7$ and $m = 3$. However, we claim that $\SL_2(\F_7)$ cannot be the group of units in a ring of characteristic 2 or 4. Indeed, $\SL_2(\F_7)$ has a maximal abelian subgroup isomorphic to $\C_8$. If $R$ realizes $\SL_2(\F_7)$, then the subring of $R$ generated by $\C_8$ has unit group $\C_8$. But $\C_8$ is not the group of units in a ring of characteristic 2 or 4 (see \cite{CL15}).

    Finally, we must prove that $G = \SL_2(\F_3)$ is not realizable in characteristic $c = 3$. In this case, $|J| \in \{1, 2, 8, 24\}.$ Thus, if $G$ is the group of units in a ring $R$ of characteristic 3, we must have $|J| = 1$. Arguing as above, this implies $G = \GL_n(\F_3)$ has proper normal subgroup $\SL_n(\F_3)$. However, $\SL_n(\F_3)$ cannot be $\Q_8$ or $\C_2$, so $n = 1$ and $G = \C_2$, a contradiction.\end{proof}
In the next proposition we prove that $\SL_2(\F_3)$ is not the group of units in a finite ring. Our proof relies on certain computations done in SageMath (\cite{sagemath}) with GAP (\cite{GAP4}). Our computations for the characteristic 2 case may be found in \cite[\href{https://cocalc.com/klockrid/main/SL2F3/files/SL(2,3)-Char-2.ipynb}{SL(2,3)-Char-2.ipynb}]{ourcode}; our computations for the characteristic 4 case may be found in \cite[\href{https://cocalc.com/klockrid/main/SL2F3/files/SL(2,3)-Char-4.ipynb}{SL(2,3)-Char-4.ipynb}]{ourcode}. In the argument below, there are two places where we have a list of cases to consider, and our code is used to verify our claims about these cases by brute force.

\begin{prop}
    The group $\SL_2(\F_3)$ is not the group of units in a finite ring.
\end{prop}
\begin{proof}

    At this point, we know that if $R$ is a finite ring with group of units $G = \SL_2(\F_3)$, then it must have characteristic $t = 2$ or $t = 4.$ Without loss of generality, we may assume that $R$ is generated by its units. The inclusion map $G \longrightarrow R$ extends to a surjective ring homomorphism \[ F\colon \Z_t[G] \longrightarrow R.\] Hence, $R$ has the form $\Z_t[G]/I$ for some two-sided ideal $I$, and the natural map \[G \longrightarrow \Z_t[G]/I\] is an embedding.
    
   Our group $G$, of order 24, is also known as the binary tetrahedral group. The quaternion group $\Q_8 = \{\pm 1, \pm i, \pm j, \pm k\}$ is a proper normal subgroup of $G$, and $G$ is generated by $\Q_8$ and an element $c \in G\setminus \Q_8$ of order 3. (Note that in the group ring $\Z_t[G]$, the element $-1 \in \Q_8$ of order 2 is not the same as the additive inverse of the multiplicative identity $1 \in R$.)

    First, we consider the case $t = 2$. We used SageMath/GAP to verify that $A = 1 + cj + (-c)$ and $B = j + c + ci$ are units in $\Z_2[G]$ of order dividing 8. Hence, there are elements $x, y \in \Q_8$ such that $A + x$ and $B + y$ are elements of $I$. However, again using SageMath/GAP, we checked that in all cases the natural map $G \longrightarrow \Z_2[G]/(A + x, B + y)$ has a nontrivial kernel, contradicting the fact that the natural map $G \longrightarrow \Z_2[G]/I$ is an embedding.

    For the case $t = 4$, we proceed similarly. Suppose $R = \Z_4[G]/I$ realizes $G$ and the natural map $G \longrightarrow \Z_4[G]/I$ is an embedding. Since $i^2$ is the unique element of order 2 in $G = R^\times$, we must have $1 + i^2 = 0$ in $R$. Thus, $1 + i^2 \in I$. Since $(i + j + k)^2 = 1,$ we have that $i + j + k \pm 1 \in I.$ For any $r \in G$, $(1 + 2r)^2 = 1$, so either $2r = 0$ or $2r = 2$ mod $I$. But since $r$ is a unit and $2 \neq 0$ in $R,$ we must have $2r = 2$ mod $I$ for all $r \in G$. Now, let $x = 1 + c + i.$ Then, $$\begin{aligned} x^3 &= 2 - c - i + 1 - c^2 -ci-cj-2c-i+c^2i+c^2j+c^2k-ck\\ &= 3 - 2i - 3c -c(i+j+k)+c^2(i+j+k-1) \\ &= 1+ (c^2-c)(i+j+k-1).\end{aligned}$$ If $i + j + k + 1 \in I,$ we obtain $x^3 = 1 + 2(c^2 - c) = 1 +2 - 2 = 1,$ so $x^3 = 1$ and $x$ has order 1 or 3. If $i + j + k -1 \in I,$ then $x$ again has order 1 or 3. So, $I$ contains the ideal $J = (1 + i^2, 1 + j + k + \epsilon, x + v)$ where $\epsilon \in \{1, -1\}$ and $v \in G$ has order $1$ or $3$. We checked all cases using SageMath/GAP, and in all cases the natural map $G \longrightarrow \Z_4[G]/J$ is not an embedding.
\end{proof}

\subsection{Rings of characteristic zero} \label{sl2char0}

We are unable to completely determine which special linear groups $\SL_2(\F_q)$ are the group of units in a ring of characteristic zero, though we have interesting partial results. It is instructive to look at the first few cases. The group $\SL_2(\F_q)$ is not realizable in characteristic zero when $q$ is even by Corollary \ref{slcor} above. The group $\SL_2(\F_3)$ is the group of units in the Hurwitz Integers. We do not know whether the binary icosahedral group $\SL_2(\F_5)$ is the group of units in a ring of characteristic zero. The obvious ring to try, by analogy with the Hurwitz Integers, is the subring of the real quaternions generated by $\SL_2(\F_5)$, but this ring has units of infinite order.

Fortunately, an infinite family of possibilities may be eliminated. The idea is as follows. If $G = \SL_2(\F_q)$ (where $q$ is odd) is realizable in characteristic zero, then so are its maximal abelian subgroups (see \cite[2.2]{CL19}). It is known which cyclic groups are realizable in characteristic zero, so if $G$ has a maximal abelian subgroup that is cyclic and not realizable in characteristic zero, then $G$ is not realizable in characteristic 0. Consider, for example, $\SL_2(\F_7)$. It contains $\C_8$ as a maximal abelian subgroup. Since $\C_8$ is not realizable in characteristic zero, neither is $\SL_2(\F_7).$

For any field $F$, let \[ \UC(F) = \left\{ \mm{x}{y}{-y}{x} \Big | \, x, y \in F\right\} \subseteq \SL_2(F). \] The set $\UC(F)$ is an abelian subgroup of $\SL_2(F)$ that we identify with the unit circle $$\{ (x, y) \in F^2 \, | \, x^2 + y^2 = 1\}$$ in $F^2$ for notational convenience. This group turns out to be cyclic and a maximal abelian subgroup of $\SL_2(\F_q).$

\begin{prop} Let $q = p^k$ be a positive odd prime power.
\begin{enumerate}
    \item If $q \equiv 1\, (4),$ then $\UC(\F_q)$ is cyclic of order $q - 1$. \label{q1}
    \item If $q \equiv 3\, (4)$, then $\UC(\F_q)$ is cyclic of order $q + 1$. \label{q3}
    \item $\UC(\F_q)$ is a maximal abelian subgroup of $\SL_2(\F_q).$ \label{maxab}
\end{enumerate}
\end{prop}
\begin{proof}

    For (\ref{q1}), suppose $q \equiv 1 \, (4)$. Then $i = \sqrt{-1} \in \F_q$ and we have $1 = x^2 + y^2 = (x + yi)(x - yi)$. The map \[F \colon \UC(\F_q) \longrightarrow \F_q^\times = \C_{q-1}\] defined by $F(x, y) = x + iy$ is an isomorphism with inverse $F^{-1}(s) = ((s + s^{-1})/2, (s - s^{-1})/(2i).$

    For (\ref{q3}), suppose $q \equiv 3\, (4)$, then $i \not \in \F_q$. Consider the map \[F \colon \UC(\F_q) \longrightarrow \F_q[i] = \F_{q^2}\] defined by $F(x, y) = x + iy.$ This map is injective since $i \not \in \F_q$, and the image of $F$ is exactly $\C_{q+1} \leq \F_{q^2}^\times$ because \[1 = x^2 + y^2 = (x + iy)(x - iy) = (x + iy)(x + iy)^q = (x + iy)^{q + 1}.\]   

    To prove (\ref{maxab}), we will show that $\UC(\F_q)$ is self-centralizing in $\SL_2(\F_q)$. (A subgroup $H$ of a group $G$ is {\bf self-centralizing} if $H$ contains its centralizer in $G$.) Take \[ A = \mm{x}{y}{z}{w} \in \SL_2(\F_q)\] in the centralizer of $\UC(\F_q)$ and take \[\mm{0}{1}{-1}{0} \in \UC(\F_q).\] Then,
    \[
    \begin{aligned}
        \mm{x}{y}{z}{w} \mm{0}{1}{-1}{0}&= \mm{0}{1}{-1}{0}\mm{x}{y}{z}{w} \\
        \mm{-y}{x}{-w}{z} &= \mm{z}{w}{-x}{-y}
    \end{aligned}
    \] from which it follows that $z = -y$ and $x = w$, so $A \in \UC(\F_q).$ Hence, $\UC(\F_q)$ is a maximal abelian subgroup of $\SL_2(\F_q)$.
\end{proof}

\begin{prop}
    Let $q = p^k$ be a positive odd prime power. If $q \equiv \pm 1\, (8)$, then $\SL_2(\F_q)$ is not realizable in characteristic zero.
\end{prop}
\begin{proof}
    Assume to the contrary that there is a ring $R$ of characteristic zero whose group of units is isomorphic to $\SL_2(\F_q).$ Then, the maximal abelian subgroup $\UC(\F_q)$ is also realizable in characteristic zero. If $q \equiv 1\, (8)$, then $q \equiv 1\, (4)$ and $\UC(\F_q) = \C_{q-1}$ with $8 \mid q - 1.$ If $q \equiv -1\, (8),$ then $q \equiv 3\, (4)$ and $\UC(\F_q) = \C_{q+1}$ with $8 \mid q + 1$. In both cases, we may conclude that a cyclic group whose order is divisible by 8 is realizable in characteristic zero; however, this is impossible by \cite{PS70}.
\end{proof}

\section{Affine linear groups} \label{sec:AGL}

In this section, we consider the affine general linear group of degree 1 over the ring $R = \Z_n$: \[ \AGL_1(\Z_n) = \Z_n \rtimes \Z_n^\times = \Hol(\Z_n).\] This is the group of affine linear transformations $x \mapsto ax + b$ where $a$ is a unit modulo $n$ and $b$ is an integer modulo $n$. In this special case, the affine general linear group is isomorphic to the holomorph of $\Z_n$. 

We begin with a few general observations. Note that we may identify $\Z_n$ with the normal subgroup $\Z_n \times \{1\}$ in $\Hol(\Z_n) = \Z_n \rtimes \Z_n^\times$.

\begin{prop}
    Let $n > 1$ be a positive integer.
    \begin{enumerate}
        \item For $n$ with prime factorization $n = p_1^{k_1} \cdots p_m^{k_m},$ we have that \[ \Hol(\Z_n) = \prod_{i =1}^m \Hol(\Z_{p_i^{k_i}})\] is the unique factorization of $\Hol(\Z_n)$ as a product of indecomposable groups. Consequently, every direct summand of $\Hol(\Z_n)$ has the form $\Hol(\Z_m)$ for some $m \mid n$. \label{holdecomp}
        \item The center of $\Hol(\Z_n)$ is trivial if $n$ is odd and cyclic of order $2$ otherwise. \label{holcenter}
        \item If $R$ is a ring with group of units $\Hol(\Z_n)$, the the characteristic of $R$ is $0, 2, 3, 4,$ or $6$; if $n$ is odd, then $R$ has characteristic $2.$\label{holchar}
        \item $\Z_n$ is a maximal abelian subgroup of $\Hol(\Z_n)$. \label{holmaxab}
        \item Every non-trivial normal subgroup of $\Hol(\Z_n)$ intersects $\Z_n$ non-trivially. Hence, if $\Hol(\Z_n)$ has a normal $p$-subgroup, then $p \mid n$. \label{holnormint}
    \end{enumerate} \label{holfacts}
\end{prop}
\begin{proof}
    First consider item (\ref{holdecomp}). As observed in \cite{miller}, if $\gcd(a, b) = 1$, then $\Hol(\Z_{ab}) = \Hol(\Z_a) \times \Hol(\Z_b).$ This gives the decomposition. By \cite[Theorem 1]{mills}, since the groups $\Z_{p_i^{k_i}}$ are indecomposable and not complete, their holomorphs are indecomposable. The claim, including the last assertion, now follows from the Krull-Schmidt Theorem.

    For item (\ref{holcenter}), suppose $(a, b)$ lies in the center of $\Z_n \rtimes \Z_n^\times$. Then, the following two commutators must be trivial:
    \[
    \begin{aligned}
        (1,1)(a,b)(1,1)^{-1}(a,b)^{-1} &= (1 + a, b)(-1, 1)(-a/b, 1/b) \\
        &= (1+a-b, b)(-a/b, 1/b) \\
        &= (1-b, 1) = (0, 1)\\
        (0,-1)(a,b)(0,-1)^{-1}(a,b)^{-1} &= (-a, -b)(0, -1)(-a/b, 1/b) \\
        &= (-a, b)(-a/b, 1/b) \\
        &= (-2a, 1) = (0, 1).
    \end{aligned}
    \] These relations force $b = 1$ and $2a = 0$. The center is therefore trivial if $n$ is odd and cyclic of order 2, generated by $(n/2, 1)$, if $n$ is even.

    Item (\ref{holchar}) follows from the previous item (cf. the proof of Proposition \ref{sl2char}).

    To prove (\ref{holmaxab}), we prove that $\Z_n$ is self-centralizing in $\Hol(\Z_n)$. Take $(a, b)$ in the centralizer of $\Z_n$. We already proved above that when $(a, b)$ commutes with $(1, 1)$, $b = 1$ is forced. Thus, $(a, b) = (a, 1) \in \Z_n$.

    Finally, we prove (\ref{holnormint}). Let $(a,b)$ be a non-trivial element of a normal subgroup $N$ of $\Hol(\Z_n).$ We obtain another element of $N$ if we take the commutator of $(1, 1)$ and $(a,b)$, as above, so $(1-b, 1) \in N$. This element is a nontrivial element of $N \cap \Z_n$ unless $b  = 1$, in which case $(a, 1)$ is a nontrivial element of $N \cap \Z_n$.
\end{proof}

\subsection{Finite rings} In this subsection, we prove Theorem \ref{aglmain}.

\begin{prop}
    If $\Hol(\Z_n)$ or $\C_2 \times \Hol(\Z_n)$ is the group of units in a finite ring with trivial Jacobson radical, then $n \mid 6$. \label{holJ0}
\end{prop}
\begin{proof}
    Suppose $R$ is a finite ring with trivial Jacobson radical and unit group $\C_2^\epsilon \times \Hol(\Z_n)$, where $\epsilon \in \{0,1\}$. 
    
    Since $J(R) = 0$, $R$ is a finite product of matrix rings of the form $\mathbf{M}_m(\F_{q^k})$ by the Artin-Wedderburn Theorem. The center of $\C_2^\epsilon \times \Hol(\Z_n)$ is a subgroup of $\C_2 \times \Z_2$, and the center of $\GL_m(\F_{q^k})$ is cyclic of order ${q^k - 1},$ so we must have $q^k = 2$ or $q^k = 3$. Thus, $\C_2^\epsilon \times \Hol(\Z_n)$ is a direct product of groups of the form $\GL_m(\F_q)$, where $q = 2$ or $q = 3$.

    Consider a (nontrivial) summand $\Hol(\Z_{p^r})$ of $\Hol(\Z_n)$ in the decomposition given in Proposition \ref{holfacts} (\ref{holdecomp}). This indecomposable group must be a summand of $\GL_m(\F_q)$ for some $m \geq 1$ and $q \in \{2,3\}.$ In this situation, $\Z_{p^r}$ is a non-trivial normal subgroup of $\GL_m(\F_q).$

    For $m > 2$, the only normal subgroups of $\GL_m(\F_q)$ are the subgroups of the center and the subgroups containing $\SL_m(\F_q)$. But since $\SL_m(\F_q)$ is non-abelian, we have that $\Z_{p^r}$ is a subgroup of $\F_2^\times$ or $\F_3^\times$, forcing $p^r = 2$ and $q = 3$.
    
    For $m = 2$, the group $\GL_2(\F_3)$ has cyclic normal subgroups of size 1 and 2 only, and the group $\GL_2(\F_2) = \Sym_3$ has cyclic normal subgroups of size 1 and 3 only. In this case, $p^r = 2$ and $q = 3$ or $p^r = 3$ and $q = 2$ is forced.

    For $m = 1$, $\GL_1(\F_2)$ and $\GL_1(\F_3)$ are cyclic of order 1 and 2, respectively. In this case, $p^r = 2$ and $q = 3$.

    Taken together, the above cases imply that $p^r = 2$ or $p^r = 3$, so $n \mid 6$. 
\end{proof}

\begin{prop}
    If $\Hol(\Z_n)$ is realizable in characteristic $3,$ then $n = 2$ or $n = 6$.
\end{prop}
\begin{proof}
    Suppose $\Hol(\Z_n)$ is realizable in characteristic 3. Then, there is a finite ring $R$ of characteristic 3 with unit group $\Hol(\Z_n)$.
    
    If $J(R) = 0$, then $n \mid 6$ by Proposition \ref{holJ0}. Since $R$ has characteristic 3, the center of $R^\times$ is nontrivial, so $R^\times$ cannot be $\Hol(\Z_1) = \C_1$ or $\Hol(\Z_3) = \Sym_3$. Thus, $n = 2$ or $n = 6$.

    If $J(R) \neq 0$, then $1 + J$ is a normal 3-subgroup of $\Hol(\Z_n)$, so $3 \mid n$ by Proposition \ref{holfacts} (\ref{holnormint}). Further, by Proposition \ref{holfacts} (\ref{holmaxab}), $\Z_n$ is realizable in characteristic 3. According to \cite[Theorem 1]{PS70}, a ring of characteristic 3 with cyclic group of units must be a product of copies of rings of the form $\F_{3^k}$ and/or $\Z_3[x]/(x^2)$. The only such ring whose unit group is cyclic of order divisible by 3 is $\Z_3[x]/(x^2)$, whose unit group is cyclic of order 6. So, if $J(R) \neq 0$, then $n = 6$.
\end{proof}

\begin{prop}
    Let $m > 0$ be odd. Every normal $2$-subgroup of \[\Hol(\Z_{4m}) = \Hol(\Z_4) \times \Hol(\Z_m) = \D_8 \times \Hol(\Z_m)\] has the form $N \times \{1\}$, where $N$ is a normal subgroup of $\D_8$. Similarly, every normal $2$-subgroup of $\Hol(\Z_{2m})$ is a subgroup of $\Hol(\Z_2)$. \label{holNSGs}
\end{prop}

\begin{proof}
    Let $H$ be a normal 2-subgroup of $G = \Z_{4m} \rtimes \Z_{4m}^\times$ with $m > 0$ odd. Take $(a, b) \in H$ and $(x, y) \in G.$ Then, as in the proof of Proposition \ref{holfacts} (\ref{holnormint}), \[ (1,1)(a,b)(-1,1)(a,b)^{-1} = (1-b, 1) \in H.\] Since the order of this element is a power of 2 and $m$ is odd, we must have $4(1-b) \equiv 0\, (4m),$ which implies $b \equiv 1\, (m).$ So $b = 1 + \epsilon m$ for some $\epsilon \in \{0,1,2,3\}$. Since $(a, b)$ has order a power of 2, we must also have, for some $k$, $b^{2^k} \equiv 1 \, (4m)$ and hence $b^{2^k} \equiv 1\, (2).$ This forces $\epsilon^{2^k}m^{2^k} \equiv 0\, (2),$ so $\epsilon$ is even because $m$ is odd. Thus, $b = 1$ or $b = 1 + 2m$ and $b^2 = 1$.

   Now $(a,b)(a,b) = (a(b+1), 1)$ has order a power of 2, so: \[
   \begin{aligned}
       4a(b+ 1) &\equiv 0\, (4m) \\
       a(b + 1) &\equiv 0\, (m) \\
       2a &\equiv 0\, (m) \\
       a &\equiv 0\, (m).
   \end{aligned}
    \] This proves that $H$ is a subgroup of $\langle m \rangle \rtimes \langle 1 + 2m \rangle$, but this subgroup is exactly the summand $\D_8 = \Hol(\Z_4)$ of $\Hol(\Z_n).$ This completes the proof of the first statement.

    The proof of the second statement is similar (and easier).
\end{proof}

\begin{prop}
    If $\Hol(\Z_n)$ is realizable in characteristic $2$ or $4$, then $n \mid 12$. However, $\Hol(\Z_1)$ and $\Hol(\Z_3) = \Sym_3$ are not realizable in characteristic $4$.
\end{prop}
\begin{proof}
    Let $R$ be a finite ring of characteristic 2 or 4 with unit group $\Hol(\Z_n)$ and Jacobson radical $J$. We already proved that if $J = 0$, then $n \mid 6$.
    
    Suppose $J \neq 0$. Then, $1 + J$ is a normal 2-subgroup of $R^\times$ since $R$ has characteristic 2 or 4. By Proposition \ref{holfacts} (\ref{holnormint}), $n$ is even. So when $n$ is odd, we have $J = 0$ and $n \mid 6$.
    
    Note that $8 \nmid n$ since $\Z_n$ is also realizable by Proposition \ref{holfacts} (\ref{holmaxab}), and $\Z_{8k}$ is not realizable in characteristic 2 or 4 by \cite[Theorem 1]{PS70}.

    If $n = 2m$ with $m$ odd and $J \neq 0$, then $1 + J$ is cyclic of order 2. By Proposition \ref{holNSGs}, we have \[ R/J = \Hol(\Z_n)/\Hol(\Z_2) = \Hol(\Z_m).\] This means $m \mid 6$, and so $n \mid 12$.

    If $n = 4m$ with $m$ odd and $J \neq 0$, then $1 + J$ is a normal subgroup of $\Hol(\Z_4) = \D_8$ of size 2 or 4. Again using Proposition \ref{holNSGs}, in the first case, $R/J = \C_2 \times \Hol(\Z_{2m})$, and $2m \mid 6$, implying $n \mid 12$. In the second case, $R/J = \C_2 \times \Hol(\Z_m)$, and $m \mid 6$, implying $n \mid 24$. But since $8 \nmid n$, we again have $n \mid 12$.

    The groups $\Hol(\Z_1)$ and $\Hol(\Z_3) = \Sym_3$ are not realizable in characteristic 4 since they have trivial centers.
\end{proof}

\begin{prop} If $\Hol(\Z_n)$ is realizable in characteristic 6, then $n = 2$ or $n = 6$.
\end{prop}
\begin{proof}
    We need only prove that $\Hol(\Z_n)$ is not realizable in characteristic 6 for $n = 1, 3, 4$. For $n = 1, 3$, $\Hol(\Z_n)$ has trivial center, but the group of units in a ring of characteristic 6 has a non-trivial center. The group $\Hol(\Z_4) = \D_8$ is not realizable in characteristic 6 by \cite{CL17d}.
\end{proof}

Taken together, the propositions in this section show that the only groups of the form $\Hol(\Z_n)$ that can occur as the group of units in a finite ring are those listed in Theorem \ref{sl2main}. We leave it to the reader to check that these groups are indeed realized, in the given characteristic, by the rings listed in the statement of the theorem.

We conclude this subsection by completing the proof of Theorem \ref{neat}. In \S 1, we explained why the first three statements are equivalent, so our remaining task is to prove that the last two statements are equivalent to the first. 

\begin{proof}[Proof of Theorem \ref{neat}]
   To prove that (\ref{equiv:1}) and (\ref{equiv:4}) are equivalent, we must prove that $n$ is a divisor of 12 if and only if $\Hol(\Z_n)$ is a direct product of dihedral groups.

    Recall that the dihedral group $\D_{2n}$ is indecomposable unless $n=2k$ with $k$ odd, in which case $\D_{2n}\cong \D_2\times \D_n$ (see \cite[2.1 (C)]{CL17d}). From this if follows that a finite group $G$ is a direct product of dihedral groups if and only if each factor in a direct product decomposition of $G$ into indecomposable groups is itself dihedral. Thus, by Proposition \ref{holfacts} (\ref{holdecomp}), to prove that (\ref{equiv:1}) and (\ref{equiv:4}) are equivalent it suffices to show that, for a prime power $p^k$, $\Hol(\Z_{p^k})$ is isomorphic to a dihedral group if and only if $p^k\in\{2,3,4\}$.
    
    The if direction is clear. For the only if direction, first consider the case where $p$ is odd. Assume $G=\Hol(\Z_{p^k})$ is isomorphic to a dihedral group.  This implies that $G$ has a cyclic (and thus abelian) subgroup $H$ of index 2. Because $H$ has index $2$, for any $x\in G$ we have that $x^2\in H$. However, because $\Z_{p^k}\leq G$ is a cyclic group of odd order, every element in $\Z_{p^k}$ is the square of some element in $G$ and thus $\Z_{p^k}\leq H$. By Proposition \ref{holfacts} (\ref{holmaxab}), we have that $\Z_{p^k}$ is a maximal abelian subgroup of $G$, and so in fact $H =\Z_{p^k}$. Since $H$ has index 2, $|H|=|G|/2=p^k\varphi(p^k)/2$ where $\varphi$ is the Euler totient function. From this, it follows that $p^k=p^k\varphi(p^k)/2$ and hence $\varphi(p^k)=2$. Now, since $p$ is odd, we have $p^k=3$.

    For the even primary case, assume that $G =\Hol(\Z_{2^k})$ is isomorphic to a dihedral group. As above, there is some cyclic $H\leq G$ with index $2$. If $\Z_{2^k}\leq H$, then we get that $\varphi(2^k)=2$ as above, in which case $2^k=4$. If instead we have that $\Z_{2^k}\not\leq H$, then there must be some generator $g$ of $\Z_{2^k}$ such that $g\in \Z_{2^k}\setminus H$. Because $g$ is a generator of $\Z_{2^k}$, it necessarily has order $2^k$; however, since $G$ is dihedral and $g$ lies outside of $H$, the order of $g$ must be $2$, and so $2^k=2$. 

    Finally, we prove that (\ref{equiv:1}) and (\ref{equiv:5}) are equivalent. Observe that $|\Hol(\Z_n)|=n\,\varphi(n)$. Simple calculations show that $n\,\varphi(n)$ is a divisor of $48$ for all $n\mid 12$. Because $\varphi(n)$ is even for all $n\geq 3$, if $n\,\varphi(n)\mid 48$ it follows that $n\mid 24$. Note that every divisor of $24$ except for $8$ and $24$ are divisors of $12$, thus the equivalence of (\ref{equiv:1}) and (\ref{equiv:5}) now follows from the fact that $8\varphi(8)=32$ and $24\varphi(24)=192$ do not divide $48$. 
 \end{proof}

\subsection{Rings of characteristic zero} \label{aglchar0} Suppose $R$ is a ring of characteristic zero with group of units $\Hol(\Z_n)$. Then, since $\Z_n$ must also be realizable in characteristic zero, $n$ can be neither odd nor divisible by 8. If $n = 4m$ with $m > 0$ odd, then $\Hol(\Z_n) = \D_8 \times \Hol(\Z_m)$. The argument given in \cite[\S 4]{CL17d} may be used, almost verbatim, to prove that $\Hol(\Z_{n})$ is not realizable in characteristic zero in this case. Hence, we must have $n = 2m$ with $m$ odd.

The rings $\Z$ and $\Z \times \mathbf{M}_2(\F_2)$ have characteristic 0 and realize $\Hol(\Z_2)$ and $\Hol(\Z_6)$, respectively. The first open case is $\Hol(\Z_{10}) = \C_2 \times \Hol(\Z_5)$. The group $\Z_{2m}$ is realizable in characteristic zero for any odd integer $m$, so one does not obtain any information using Proposition \ref{holfacts} (\ref{holmaxab}).

\bibliographystyle{alpha}
\bibliography{mainrev.bib}
\end{document}